\documentclass[a4paper,12pt]{article}

\usepackage{amsfonts}
\usepackage{amscd,color}
\usepackage{amsmath,amsfonts,amssymb,amscd}
\usepackage{indentfirst,graphicx,epsfig}
\usepackage{graphicx}
\input{epsf}
\usepackage{epstopdf}
\usepackage{caption}
\usepackage{ifpdf}
\usepackage{fullpage}
\usepackage{epsfig}
\usepackage{latexsym}
\usepackage{caption}
\usepackage{amsmath,amsthm}
\usepackage{amsfonts}
\usepackage{amssymb}
\usepackage{graphicx}
\usepackage{graphics}
\usepackage{bm}
\usepackage{color}
\usepackage{subfigure}
\usepackage{diagbox}
\graphicspath{{\figs}}

\newcommand{\tabincell}[2]{\begin{tabular}{@{}#1@{}}#2\end{tabular}}

\makeatletter

\newcommand{\Rmnum}[1]{\expandafter\@slowromancap\romannumeral #1@}
\makeatother
\makeatletter
\let\@fnsymbol\@arabic
\makeatother

\begin{document}
\newtheorem{theorem}{Theorem}[section]
\newtheorem{observation}[theorem]{Observation}
\newtheorem{corollary}[theorem]{Corollary}
\newtheorem{algorithm}[theorem]{Algorithm}
\newtheorem{problem}[theorem]{Problem}
\newtheorem{question}[theorem]{Question}
\newtheorem{lemma}[theorem]{Lemma}
\newtheorem{proposition}[theorem]{Proposition}

\newtheorem{definition}[theorem]{Definition}
\newtheorem{guess}[theorem]{Conjecture}
\newtheorem{claim}[theorem]{Claim}
\newtheorem{example}[theorem]{Example}
\makeatletter
  \newcommand\figcaption{\def\@captype{figure}\caption}
  \newcommand\tabcaption{\def\@captype{table}\caption}
\makeatother

\newtheorem{acknowledgement}[theorem]{Acknowledgement}

\newtheorem{axiom}[theorem]{Axiom}
\newtheorem{case}[theorem]{Case}
\newtheorem{conclusion}[theorem]{Conclusion}
\newtheorem{condition}[theorem]{Condition}
\newtheorem{conjecture}[theorem]{Conjecture}
\newtheorem{criterion}[theorem]{Criterion}
\newtheorem{exercise}[theorem]{Exercise}
\newtheorem{notation}[theorem]{Notation}
\newtheorem{solution}[theorem]{Solution}
\newtheorem{summary}[theorem]{Summary}
\newtheorem{fact}[theorem]{Fact}

\newcommand{\pp}{{\it p.}}
\newcommand{\de}{\em}
\newcommand{\mad}{\rm mad}

\newcommand*{\QEDA}{\hfill\ensuremath{\blacksquare}}
\newcommand*{\QEDB}{\hfill\ensuremath{\square}}

\newcommand{\qf}{Q({\cal F},s)}
\newcommand{\qff}{Q({\cal F}',s)}
\newcommand{\qfff}{Q({\cal F}'',s)}
\newcommand{\f}{{\cal F}}
\newcommand{\ff}{{\cal F}'}
\newcommand{\fff}{{\cal F}''}
\newcommand{\fs}{{\cal F},s}
\newcommand{\g}{\gamma}
\newcommand{\wrt}{with respect to }

\title{Extremal graphs and classification of planar graphs by MC-numbers\footnote{Supported by NSFC No.11871034 and 11531011.}}

\renewcommand{\thefootnote}{\arabic{footnote}}

\author{\small Yanhong Gao, Ping Li, Xueliang Li\\
\small Center for Combinatorics and LPMC\\
\small  Nankai University, Tianjin 300071, China\\
\small Email: gyh930623@163.com, qdli\underline{ }ping@163.com, lxl@nankai.edu.cn\\
}

\date{}
\maketitle

\begin{abstract}
An edge-coloring of a connected graph $G$ is called a {\em monochromatic connection coloring} (MC-coloring for short) if any two vertices of $G$ are
connected by a monochromatic path in $G$.
For a connected graph $G$, the {\em monochromatic connection number} (MC-number for short) of $G$, denoted by $mc(G)$, is the maximum number of colors that ensure $G$ has a
monochromatic connection coloring by using this number of colors.
This concept was introduced by Caro and Yuster in 2011. They proved that $mc(G)\leq m-n+k$ if $G$ is not a $k$-connected graph. In this paper
we depict all graphs with $mc(G)=m-n+k+1$ and $mc(G)= m-n+k$ if $G$ is a $k$-connected but not $(k+1)$-connected graph.
We also prove that $mc(G)\leq m-n+4$ if $G$ is a planar graph, and classify all planar graphs by their monochromatic connectivity numbers.\\[2mm]
\end{abstract}

\baselineskip16pt

\section{Introduction}

All graphs considered in this paper are simple, finite and undirected.
We use $\kappa(G)$ to denote the connectivity of $G$, and $\chi(G)$ to denote the chromatic number of $G$.
A planar graph is an {\em outerplanar graph} if it has an embedding with every vertex on the boundary of the unbounded face.
Generally, the notation $[k]$ refers to the set $\{1,2,\cdots,k\}$ of integers.
For $k$ pairwise disjoint vertex-sets $U_1,\cdots,U_k$ of $G$, we say $U_1,\cdots,U_k$ to form a complete multipartite graph
if every vertex of $U_i$ connects every vertex of $U_j$ in $G$ for any $i\neq j$.
If there is no confusion, we always use $m$ and $n$ to denote the numbers of edges and vertices of a graph, respectively.
Sometimes, we also use $e(G)$ and $|G|$ to denote the numbers of edges and vertices of graph $G$, respectively.
For a graph $G$, $d_{G}(v)$ is defined as the degree of a vertex $v$, which is the number of neighbors of $v$ in $G$.
If $d_{G}(v)=t$, then we call $v$ a {\em $t$-degree vertex} of $G$.
A forest is called a {\em linear forest} if every component of the forest is either a path or a vertex.
We use $P_n,C_n,S_n,K_n^1$ to denote a path with $n$ vertices, a cycle with $n$ edges, a star with $n$ edges and a graph obtained from $K_n$ by removing one edge, respectively.
Analogically, a $k$-path or a $k$-cycle is a path or a cycle with $k$ edges.
For an edge $e=xy$ of $G$, $G/e$ denotes a graph obtained from $G$ by deleting $e$ and then identifying $x$ and $y$,
which means replacing the two vertices by a {\em new vertex} incident to all the edges which were incident with either $x$ or $y$ in $G$.
Suppose $G$ and $H$ are vertex-disjoint graphs.
Then let $G\vee H$ denote the {\em join} of $G$ and $H$,
which is obtained from $G$ and $H$ by adding an edge between each vertex of $G$ and every vertex of $H$,
and let $G+H$ denote a graph with vertex set $V(G)\cup V(H)$ and edge set $E(G)\cup E(H)$.
If $G=H$, we also denote $G+H$ by $2G$.

An edge-coloring of $G$ is a mapping from $E(G)$ to a positive integer set, say $[k]$.
A monochromatic graph is a graph whose edges are assigned the same color.
An edge-coloring of a connected graph $G$ is called a {\em monochromatic connection coloring} (MC-coloring for short)
if any two vertices of $G$ are connected by a monochromatic path in $G$, and the edge-colored graph $G$ is called {\em monochromatic connected}.
An {\em extremal monochromatic connection coloring} (extremal MC-coloring for short) of $G$ is a
monochromatic connection coloring of $G$ that uses the maximum number of colors.
For a connected graph $G$, the {\em monochromatic connection number} (MC-number for short) of $G$, denoted by $mc(G)$, is the number of colors
in an extremal monochromatic connection coloring of $G$.

Suppose $\Gamma$ is an edge-coloring of $G$ and $i$ is a color of $\Gamma(G)$.
The {\em $i$-induced subgraph} is a subgraph of $G$ induced by all the edges with color $i$.
We also call an $i$-induced subgraph a {\em color-induced subgraph}.
Suppose the $i$-induced subgraph is $F$.
If $F$ is a single edge, then we call the color $i$ and $F$ {\em trivial}.
Otherwise, they are called {\em nontrivial}.
For a subgraph $H$ of $G$, we denote $\Gamma|_H$ as the edge-coloring of $H$ with restricting the edge-coloring $\Gamma$ of $G$ to $H$.

Some properties of the MC-colorings were discussed in \cite{CY}, we list them here.
An edge-coloring of $G$ is {\em simple} if any two nontrivial color-induced subgraphs intersect in at most one vertex.
There exists a simple extremal MC-coloring for every connected graph.
Suppose $\Gamma$ is an extremal MC-coloring of $G$. Then each color-induced subgraph in $G$ is a tree.
If there are $t$ edges in a color-induced subgraph, then we call the color {\em wastes } $t-1$ colors.
Suppose $\Gamma$ is an edge-coloring of $G$ and $\mathcal{H}$ is the set of all nontrivial color-induced subgraphs.
Then $\Gamma$ wastes $w(\Gamma)=\Sigma_{H\in \mathcal{H}}(e(H)-1)$ colors.
Thus, the number of colors used in $G$ is equal to $m-w(\Gamma)$.
If $\Gamma$ is an extremal MC-coloring of $G$,
then since each color-induced subgraph is a tree, we have that
$w(\Gamma)=\Sigma_{H\in \mathcal{H}}(e(H)-1)=\Sigma_{H\in \mathcal{H}}(|H|-2)$,
and thus $mc(G)=m-\Sigma_{H\in \mathcal{H}}(|H|-2)$.

For a connected graph $G$, we can obtain an MC-coloring by coloring a spanning tree monochromatically and coloring every other edge with a trivial color.
Therefore, $mc(G)\geq m-n+2$ for every connected graph $G$.
Caro and Yuster showed the following results.

\begin{theorem}[\cite{CY}]\label{CY-1}
Let $G$ be a connected graph with $n\geq 3$.
If $G$ satisfies any of the following properties, then $md(G)=m-n+2$.
\begin{enumerate}
\item $\overline{G}$ (the complement of $G$) is $4$-connected;
\item $G$ is triangle-free;
\item $\Delta(G)<n-\frac{2m-3(n-1)}{n-3}$;
\item the diameter of $G$ is greater than or equal to three;
\item $G$ has a cut-vertex.
\end{enumerate}
\end{theorem}

\begin{theorem}[\cite{CY}]\label{CY-2}
Let $G$ be a connected graph. Then
\begin{enumerate}
\item $mc(G)\leq m-n+\chi(G)$;
\item $mc(G)\leq m-n+k+1$ if $G$ is not a $(k+1)$-connected graph.
\end{enumerate}
\end{theorem}

A graph is called {\em $s$-perfectly-connected} if it can be partitioned into $s+1$ parts $\{v\},V_1,\cdots,V_s$,
such that each $V_i$ induces a connected subgraph, $V_1,\cdots,V_s$ form a complete $s$-partite graph, and $v$ has precisely one neighbor in each $V_i$.
\begin{proposition}[\cite{CY}]\label{CY-3}
If $\delta(G)=s$, then $mc(G)\leq m-n+s$, unless $G$ is $s$-perfectly-connected, in which case $mc(G)=m-n+s+1$.
\end{proposition}

In \cite{JLY}, the authors characterized all graphs with $mc(G)=m-n+\chi(G)$.
In \cite{LL, LL1}, the authors generalized the concept of MC-coloring.
For more knowledge about the monochromatic connection of graphs, we refer to \cite{BL, CLW, GLQZ, JLW, LW2, MWYY}.
In \cite{CY}, Caro and Yuster showed that the bound of the second result is sharp, and they
studied wheel graphs, outerplanar graphs and planar graphs with minimum degree three.

Let $\mathcal{S}$ be a set of trees. Then we use $V(\mathcal{S})$ to denote $\bigcup_{T\in \mathcal{S}}V(T)$, and $|\mathcal{S}|$ to denote
the number of trees in $\mathcal{S}$. Suppose $G$ is a $k$-connected graph and $\Gamma$ is an MC-coloring of $G$.
Let $S=\{w_1,\cdots,w_k\}$ be a vertex-cut of $G$ and $A_1,\cdots, A_t$ be the components of $G-S$.
For a vertex $x\in V(A_i)$, we always use $\mathcal{T}_x$ to denote the set of nontrivial trees connecting $x$ and a vertex in $\bigcup_{j\neq i}A_j$ in this paper.
Since $x$ connects every vertex of $\bigcup_{j\neq i}A_j$ by a nontrivial tree, we have $\bigcup_{j\neq i}V(A_j)\subseteq V(\mathcal{T}_x)$.

This paper is organized as follows. In Section 2, we depict all graphs with $mc(G)=m-n+\kappa(G)+1$ and $mc(G)=m-n+\kappa(G)$, respectively.
In Section 3, we classify all planar graphs by their monochromatic connection numbers.

\section{Extremal graphs of $k$-connected graphs}

For a $k$-connected graph $G$, we know that $mc(G)\leq m-n+k+1$.
In this section, we depict all graphs with $mc(G)=m-n+k+1$ and $mc(G)=m-n+k$, respectively.
These results will be used in the next section for classification of planar graphs.

Let $\mathcal{A}_{n,k}$ be the set of graphs $K_{k-1}\vee H$, where $H$ is a connected graph with $|H|=n-k+1$ and $H$ has a cut-vertex.

\begin{theorem}\label{Ex-mc-thm1}
Suppose $k\geq 2$ and $G$ is a $k$-connected but not $(k+1)$-connected graph.
Then $mc(G)=m-n+k+1$ if and only if either $G\in \mathcal{A}_{n,k}$ or $G$ is a $k$-perfectly-connected graph.
\end{theorem}
\begin{proof}
If $G$ is a $k$-perfectly-connected graph, then by Proposition \ref{CY-3}, $mc(G)=m-n+k+1$.
If $G=K_{k-1}\vee H$ is a graph in $\mathcal{A}_{n,k}$, then let $\Gamma$ be an edge-coloring of $G$ such that
a spanning tree of $H$ is the only nontrivial tree. Then $\Gamma$ is an MC-coloring of $G$ and $\Gamma$ wastes $n-k-1$ colors.
Thus, $mc(G)=m-n+k+1$.

Next, we prove that either $G\in \mathcal{A}_{n,k}$ or $G$ is a $k$-perfectly-connected graph if $mc(G)=m-n+k+1$.
Let $\Gamma$ be an extremal MC-coloring of $G$ and $\mathcal{S}$ be the set of all non-trivial trees.
Let $S=\{w_1,\cdots,w_k\}$ be a vertex-cut and $A_1,\cdots, A_t$ be the components of $G-S$.

{\bf Case 1.} There is a component, say $A_1$, and a vertex $u$ of $A_1$, such that $V(A_1)\subseteq V(\mathcal{T}_u)$.

Let $\mathcal{T}_u=\{T_1,\cdots,T_r\}$.
Since $u$ connects every vertex of $\bigcup_{i=2}^tV(A_i)$ by a nontrivial tree in $\{T_1,\cdots,T_r\}$, we have
$\bigcup_{i\in[t]}V(A_i)\subseteq V(\bigcup_{i\in[r]}T_i)$.
Since any two trees of $\{T_1,\cdots,T_r\}$ share a common vertex $u$ and $\Gamma$ is simple, we have $\bigcup_{i\in[r]}T_i$ is a tree.
Moreover, $|V(\bigcup_{i\in[r]}T_i)\cap S|\geq r$.
Therefore, $\bigcup_{i\in[r]}T_i$ wastes at least $n-(k-r)-1-r=n-k-1$ colors.
Since $mc(G)=m-n+k+1$, we have $\mathcal{S}=\{T_1,\cdots,T_r\}$ and $|V(\bigcup_{i\in[r]}T_i)\cap S|= r$.
Thus, $|V(T_i)\cap S|=1$, say $V(T_i)\cap S=\{w_i\}$.

If $A_1=\{u\}$, then $N(u)=S$.
Since $G$ is a $k$-connected graph, we have $\delta(G)=k$.
By Proposition \ref{CY-3}, $mc(G)=m-n+k+1$ implies that $G$ is a $k$-perfectly-connected graph.

If $|A_1|\geq 2$, then $r=1$.
Otherwise, there are at least two nontrivial trees in $\mathcal{S}$. Suppose $v\in V(A_1)-u$ and $v\in V(T_1)$.
Let $w\in (\bigcup_{i=2}^tV(A_i))\cap V(T_2)$.
Then there is a nontrivial tree $T_j$ connecting $w$ and $v$.
Since $v\in V(T_j)$ and $v\notin V(T_2)$, $T_j\neq T_2$.
However, $\{u,w\}\subseteq V(T_j)\cap V(T_2)$, a contradiction.
Therefore, $\mathcal{S}=\{T_1\}$.
Since $mc(G)=m-n+k+1$, we have $|T_1|=n-k+1$.
Recall that $V(T_1)\cap S=\{w_1\}$.
Let $S'=S-w_1$.
Then $T_1$ is a spanning tree of $G-S'$.
Thus, $G-S'$ is connected and $w_1$ is a cut-vertex of $G-S'$.
Since $T_1$ is the unique nontrivial tree of $G$, we have $G[S']=K_{k-1}$ and $G=G[S']\vee (G-S')$.
Therefore, $G\in \mathcal{A}_{n,k}$.

{\bf Case 2.} For each component $A_i$ of $G-S$ and each vertex $u\in V(A_i)$, $V(A_i)-V(\mathcal{T}_u)\neq \emptyset$.

For a vertex $u$ of $A_1$, let $A=V(A_1)-V(\mathcal{T}_u)$ and $v\in A$.
Let $w\in V(A_2)$, and let $\mathcal{F}$ be the set of nontrivial trees connecting $w$ and a vertex of $A$.
Since $\Gamma$ is simple, we have $|V(\mathcal{T}_u)\cap S|\geq |\mathcal{T}_u|$ and $|V(\mathcal{F})\cap S|\geq |\mathcal{F}|$.
So, $\mathcal{T}_u$ wastes at least $n-k-|A|-1$ colors and $\mathcal{F}$ wastes at least $|A|$ colors.
Since $mc(G)=m-n+k+1$, $\mathcal{T}_u$ wastes precisely $n-k-|A|-1$ colors, $\mathcal{F}$ wastes precisely $|A|$ colors and $\mathcal{S}=\mathcal{T}_u\cup \mathcal{F}$.
That $\mathcal{F}$ wastes precisely $|A|$ colors implies that $V(A_2)\cap V(T)=\{w\}$ for each $T\in \mathcal{F}$.
Since $V(A_2)\nsubseteq V(\mathcal{T}_w)$, there is at least one vertex in $V(A_2)-V(\mathcal{T}_w)$, say $w'\in V(A_2)-V(\mathcal{T}_w)$.
Then there is no tree of $\mathcal{T}_u\cup \mathcal{F}$ that contains both $v$ and $w'$, which contradicts that $\mathcal{S}=\mathcal{T}_u\cup \mathcal{F}$.
\end{proof}

Let $\mathcal{B}^1_{n,k}$ be the set of graphs $G$ satisfying the following conditions:

$\bullet$ $G$ is $k$-connected.

$\bullet$ $V(G)$ can be partitioned into $k$ nonempty sets $\{u\},U_1,\cdots,U_{k-1}$ such that each $G[U_i\cup u]$ is connected and $U_1,\cdots,U_{k-1}$ form a complete $(k-1)$-partite graph.

$\bullet$ There is an integer $t\in[k-1]$, such that $u$ has precisely two neighbors in $U_t$ and $u$ has precisely one neighbor in $U_i$ for $i\neq t$.

$\bullet$ $G$ is neither a $k$-perfectly-connected graph nor a graph of $\mathcal{A}_{n,k}$.

Let $\mathcal{B}^2_{n,k}$ be the set of graphs $G$ satisfying the following conditions:

$\bullet$ $G$ is a $k$-connected graph, $V(G)$ can be partitioned into two parts $U,V$ such that $G[U]=K_{k-2}$, $G[V]$ is a $2$-connected but not a $3$-connected graph and $G=G[U]\vee G[V]$;

$\bullet$ $G$ is neither a $k$-perfectly-connected graph nor a graph of $\mathcal{A}_{n,k}$.

Let $\mathcal{B}^3_{n,k}$ be the set of graphs $G$ satisfying the following conditions:

$\bullet$ $G$ is a $k$-connected graph, $V(G)$ can be partitioned into two parts $U,V$ such that $G[U]=K_{k-1}^-$, $G[V]$ is a connected graph with a cut-vertex and $G=G[U]\vee G[V]$;

$\bullet$ $G$ is neither a $k$-perfectly-connected graph nor a graph of $\mathcal{A}_{n,k}$.

\begin{lemma}\label{Ex-mc-lem2}
If $G\in \mathcal{B}^1_{n,k}\cup \mathcal{B}^2_{n,k}\cup \mathcal{B}^3_{n,k}$, then $mc(G)=m-n+k$.
\end{lemma}
\begin{proof}
Let $G\in \mathcal{B}^1_{n,k}\cup \mathcal{B}^2_{n,k}\cup \mathcal{B}^3_{n,k}$.
It is easy to verify that $G$ is $k$-connected but not $(k+1)$-connected.
Since $G$ is neither a $k$-perfectly-connected graph nor a graph of $\mathcal{A}_{n,k}$,
we have $mc(G)<m-n+k+1$.

If $G\in \mathcal{B}^1_{n,k}$, then let $T_i$ be a spanning tree of $G[U_i\cup \{u\}]$ for $i\in[k-1]$.
We color the edges of $T_i$ with $i$ and color any other edges with trivial colors.
Then the edge-coloring is an MC-coloring of $G$, which uses $m-n+k$ colors.
Thus, $mc(G)= m-n+k$.

If $G\in \mathcal{B}^2_{n,k}$, then we color the edges of $G$ such that a spanning tree of $G[V]$ is
the unique nontrivial color-induced subgraph.
The edge-coloring is obviously an MC-coloring of $G$, which uses $m-n+k$ colors.
Thus, $mc(G)= m-n+k$.

If $G\in \mathcal{B}^3_{n,k}$, then let $T$ be a spanning tree of $G[V]$ and let $F$ be a $2$-path obtained by connecting one vertex of $G[V]$ and two nonadjacent vertices of $G[U]$.
We color the edges of $G$ such that $\{T,F\}$ is the set of nontrivial color-induced subgraphs.
The edge-coloring is obviously an MC-coloring of $G$, which uses $m-n+k$ colors.
Thus, $mc(G)= m-n+k$.
\end{proof}

\begin{theorem}\label{Ex-thm2}
Suppose $k\geq 3$, and $G$ is a $k$-connected but not $(k+1)$-connected graph. Then $mc(G)=m-n+k$ if and only if $G\in \mathcal{B}^1_{n,k}\cup \mathcal{B}^2_{n,k}\cup \mathcal{B}^3_{n,k}$.
\end{theorem}
\begin{proof}
If $G\in \mathcal{B}^1_{n,k}\cup \mathcal{B}^2_{n,k}\cup \mathcal{B}^3_{n,k}$, then by Lemma \ref{Ex-mc-lem2}, $mc(G)=m-n+k$.

Suppose $mc(G)=m-n+k$.
Next we prove $G\in \mathcal{B}^1_{n,k}\cup \mathcal{B}^2_{n,k}\cup \mathcal{B}^3_{n,k}$.
Let $S=\{v_1,\cdots,v_k\}$ be a vertex-cut of $G$ and $G-S$ have $r$ components $A_1,\cdots,A_r$.
Let $\Gamma$ be an extremal MC-coloring of $G$ and $u\in V(A_i)$.
Then $\Gamma$ wastes $n-k$ colors.
Since $\Gamma$ is simple, any two trees of $\mathcal{T}_u$ intersect only at $u$.
Thus, $\mathcal{T}_u$ wastes at least \begin{align}&|\bigcup_{l\neq i}A_l|+|V(\mathcal{T}_u)\cap V(A_i)|+|V(\mathcal{T}_u)\cap S|-1-|\mathcal{T}_u|\\
&=n-k-|V(A_i)-V(\mathcal{T}_u)|+(|V(\mathcal{T}_u)\cap S|- |\mathcal{T}_u|)-1
\end{align} colors.
\begin{claim}\label{Ex-clm0}
Suppose $U\subseteq V(A_1)$.
Then $\bigcup_{w\in U}\mathcal{T}_w$ wastes at least $|U|+|\bigcup_{l=2}^rA_l|-1$ colors.
\end{claim}
\begin{proof}
Let $U=\{a_1,\cdots,a_q\}$ and let $\mathcal{F}_i=\mathcal{T}_{a_i}-\bigcup_{l=1}^{i-1}\mathcal{T}_{a_l}$.
Suppose $\mathcal{F}_i$ contains $c_i$ vertices of $U$. Then $\sum_{i\in[q]}c_i\geq q=|U|$.
Since each tree of $\mathcal{F}_i$ connects one vertex of $S$ and one vertex of $\bigcup_{l=2}^rA_l$,
$\mathcal{F}_i$ wastes at least $c_i$ colors if $c_i\neq 0$.
Since $\mathcal{F}_i=\mathcal{T}_{a_1}$ wastes at least $|\bigcup_{l=2}^rA_l|+c_1-1$ colors by equality (1),
$\bigcup_{w\in U}\mathcal{T}_w$ wastes at least
\begin{align*}\sum_{i\in[q]}w_i&\geq |\bigcup_{l=2}^rA_l|+c_1-1+\Sigma_{i=2}^qc_i
=|\bigcup_{l=2}^rA_l|-1+\sum_{i\in[q]}c_i\geq
|\bigcup_{l=2}^rA_l|+|U|-1
\end{align*} colors.
\end{proof}

\begin{claim}\label{Ex-clm00}
If $T$ is a $2$-path of $G$, then the two leaves of $T$ are nonadjacent.
\end{claim}
\begin{proof}
Suppose the two leaves of $T$ are adjacent.
Then recolor every edge of $T$ by a trivial color.
It is easy to verify that the new coloring is an MC-coloring of $G$. However, the new coloring wastes less colors,
a contradiction to the assumption that $\Gamma$ is extremal.
\end{proof}

{\bf Case 1.} There is a component, say $A_1$, and a vertex $u$ of $A_1$ such that $A_1\subseteq V(\mathcal{T}_u)$.

Let $\mathcal{T}_u=\{T_1,\cdots,T_t\}$ and $B=\bigcup_{l=2}^rV(A_l)$.
Here $T_i$ is a tree colored $i$.
Each $T_i$ contains at least one vertex of $S$.

{\bf Case 1.1.} $V(A_1)=\{u\}$.

Since $S$ is a vertex-cut of order $k$ and $G$ is a $k$-connected graph,
$u$ connects every vertex of $S$, i.e., $S=N(u)$.

If there is a tree of $\mathcal{T}_u$, say $T_t$, which contains at least two vertices of $S$,
then by (2), $\mathcal{T}_u$ wastes at least $n-k$ colors.
Since $mc(G)=m-n+k$, $\mathcal{T}_u$ wastes precisely $n-k$ colors.
Thus, $T_t$ contains precisely two vertices of $S$ (say $v_t,v_{t+1}$), and $T_l$ contains precisely one vertex
of $S$ for $l\in[t-1]$ (say $v_l$). Therefore, $\mathcal{T}_u$ is the set of all nontrivial trees of $G$.
Since $\Gamma$ is simple, any two trees of $\mathcal{T}_u$ share a common vertex $u$.
Let $U_i=V(T_i)-\{u\}$ for $i\in [t]$ and $U_i=\{v_{i+1}\}$ for $t+1\leq i\leq k-1$.
Then $\{u,U_1,\cdots,U_{k-1}\}$ is a partition of $V(G)$ and each $G[U_i\cup \{u\}]$ is connected.
Moreover, $|U_i\cap N(u)|=1$ for $i\neq t$ and $|U_t\cap N(u)|=2$.
Since there is no nontrivial tree connecting a vertex of $U_i$ and a vertex of $U_j$ if $i\neq j$,
$U_1,\cdots,U_{k-1}$ form a complete multipartite graph.
Since $mc(G)\neq m-n+k+1$, by Theorem \ref{Ex-mc-thm1}, $G$ is neither a $k$-perfectly-connected graph nor a graph of $\mathcal{A}_{n,k}$.
Thus, $G\in \mathcal{B}^1_{n,k}$.

If every tree of $\mathcal{T}_u$ contains precisely one vertex of $S$, say $V(T_i)\cap S=\{v_i\}$ for $i\in[t]$.
Then $\mathcal{T}_u$ wastes $n-k-1$ colors.
Thus, there is a nontrivial tree $T$ that wastes one color, i.e., $T$ is a $2$-path.
So, $\mathcal{T}_u\cup \{T\}$ is the set of all nontrivial trees of $G$.
Since $T$ is a $2$-path, by Claim \ref{Ex-clm00}, the two leaves of $T$ are nonadjacent.
Let $U_i=V(T_i)-\{u\}$ for $i\in[t]$ and $U_i=\{v_i\}$ for $t+1\leq i\leq k$.
Since $\Gamma$ is simple, the two leaves of $T$ can not appear in the same set $U_i$.
Thus, there are two different integers $i,j$ of $[k]$ such that one leaf of $T$ is in $U_i$ and the other leaf is in $U_j$.
Then $U_1,\cdots,U_i\cup U_j,\cdots,U_k$ form a complete $(k-1)$-partite graph.
Since $mc(G)\neq m-n+k+1$, by Theorem \ref{Ex-mc-thm1}, $G$ is neither a $k$-perfectly-connected graph nor a graph of $\mathcal{A}_{n,k}$.
Recalling the definition of $\mathcal{B}^1_{n,k}$, we get $G\in \mathcal{B}^1_{n,k}$.

{\bf Case 1.2.} $t=1$.

From the assumption, $\bigcup_{i\in[r]}V(A_i)\subseteq V(T_1)$.
Then $T_1$ wastes $n-k+|V(T_1)\cap S|-2$ colors.
Since $\Gamma$ wastes $n-k$ colors, either $T_1$ is the only nontrivial tree and $|V(T_1)\cap S|=2$, or $|V(T_1)\cap S|=1$ and there is a $2$-path $F$ such that $\{F,T_1\}$ is the set of all nontrivial trees.
Let $V=V(T_1)$ and $U=V(G)-V$.

If $|V(T_1)\cap S|=2$,
then since $T_1$ is the unique nontrivial tree of $\Gamma$, we have $G[U]=K_{k-2}$ and $G=G[U]\vee G[V]$.
Since $S$ is a vertex-cut with $|S|=k$, $V(T_1)\cap S$ is a vertex-cut of $G-U$, i.e., $G[V]$ is a $2$-connected but not $3$-connected graph.
Since $G$ is neither a $k$-perfectly-connected graph nor a graph of $\mathcal{A}_{n,k}$, we have $G\in\mathcal{B}^2_{n,k}$.

If $|V(T_1)\cap S|=1$, then suppose $F=x_1e_1ye_2x_2$ and $V(T_1)\cap S=\{w\}$.
If, by symmetry, $x_1\in V(T_1)$, then $V(F)\cap V(T_1)=\{x_1\}$.
Let $w'\in V(T_1)-\{x_1\}$.
Then $w'x_2$ is a trivial edge of $G$.
Let $T=T_1\cup w'x_2$ and let $\Gamma'$ be an edge-coloring of $G$ such that $T$ is the only nontrivial tree of $G$.
Then $\Gamma'$ is an extremal MC-coloring of $G$ with $|V(T)\cap S|=2$, this case has been discussed above.
If $\{x_1,x_2\}\cap V(T_1)=\emptyset$, then $G[U]=K_{k-1}^-$ and $G=G[U]\vee G[V]$.
Moreover, $G[V]$ is a connected graph with a vertex-cut $w$.
Thus, $G\in\mathcal{B}^3_{n,k}$.

{\bf Case 1.3.} $|A_1|\geq 2$ and $t\geq 2$.

If $|A_1|\geq 3$, then there are two trees of $\mathcal{T}_u$, say $T_1,T_2$, such that either $|V(T_1)\cap V(A_1)|\geq 3$ or $|V(T_1)\cap V(A_1)|=|V(T_2)\cap V(A_1)|=2$.
Let $w_i\in V(T_i)\cap B$ for $i\in[2]$.
If $|V(T_1)\cap V(A_1)|\geq 3$, then there are trees of $\mathcal{T}_{w_2}-\mathcal{T}_u$ connecting $w_2$ and $V(T_1)\cap V(A_1)-\{u\}$.
It is obvious that $\mathcal{T}_{w_2}-\mathcal{T}_u$ wastes at least two colors.
Since $\mathcal{T}_u$ wastes at least $n-k-1$ colors,
$\mathcal{T}_{w_2}\cup \mathcal{T}_u$ wastes at least $n-k-1+2=n-k+1$ colors, which contradicts that $\Gamma$ is an extremal MC-coloring of $G$.
If $|V(T_1)\cap V(A_1)|=|V(T_2)\cap V(A_1)|=2$, say $\{z_i\}=V(T_i)\cap V(A_1)-\{u\}$ for $i\in[2]$.
Then there is a nontrivial tree $F_1$ connecting $w_1,z_2$, and a nontrivial tree $F_2$ connecting $w_2,z_1$.
Since $\Gamma$ is simple, we have $F_1\neq F_2$.
Since $\{F_1,F_2\}\cap \mathcal{T}_u=\emptyset$, $\{F_1,F_2\}\cup \mathcal{T}_u$ wastes at least $n-k+1$ colors, a contradiction.
Therefore, $|A_1|=2$.
Let $V(A_1)=\{z,u\}$ and let $T_1$ contain $z,u$.
Then $V(T_i)\cap S=\{u\}$ for $i\geq 2$.

Since $t\geq 2$, we have $B-V(T_1)\neq \emptyset$.
Then $z$ connects every vertex of $B-V(T_1)$ by a nontrivial tree, $\mathcal{T}_z-\mathcal{T}_u$ is not an empty set.
It is obvious that $\mathcal{T}_u$ wastes at least $n-k-1$ colors and $\mathcal{T}_z-\mathcal{T}_u$ wastes at least one color.
Since $mc(G)=m-n+k$, $\mathcal{T}_u$ wastes precisely $n-k-1$ colors and $\mathcal{T}_z-\mathcal{T}_u$ wastes precisely one color.
Therefore, $\mathcal{T}_z-\mathcal{T}_u$ just has one member, and the member is a $2$-path (call the $2$-path $F$, then $\mathcal{T}_z-\mathcal{T}_u=\{F\}$).
So, $|B-V(T_1)|=1$ and $t=2$.
Then $\mathcal{T}_u=\{T_1,T_2\}$ and $\mathcal{S}=\{F,T_1,T_2\}$ is the set of all nontrivial trees.
We can also get that each tree of $\mathcal{S}$ intersects $S$ at only one vertex.
So, $F$ and $T_2$ are $2$-paths.

Let $\Gamma'$ be an edge-coloring of $G$ obtained from $\Gamma$ by recoloring $T'=T_1\cup F$ with $1$ and recoloring any other edges with trivial colors.
Then the new coloring is also an MC-coloring of $G$.
Since $\Gamma'$ wastes $n-k$ colors, $\Gamma'$ is an extremal MC-coloring of $G$.
Then $T'$ is the unique nontrivial tree of $\Gamma'$ and $|V(T')\cap S|=2$, this case has been discussed in Case 1.2.

{\bf Case 2.} For each $i\in[r]$ and each $u\in A_i$, $V(A_i)-V(\mathcal{T}_u)\neq \emptyset$
(then each $A_l$ has order at least two).

If there is an integer $i\in[r]$ such that $|\bigcup_{l\neq i}A_l|\geq 3$, then let $u\in V(A_i)$ and let $A'=V(A_i)-V(\mathcal{T}_u)$.
Then $\mathcal{T}_u$ wastes at least $n-|A'|-k-1$ colors.
By Claim \ref{Ex-clm0}, $\bigcup_{w\in V(A')}\mathcal{T}_w$ wastes at least $|A'|+|\bigcup_{l\neq i}A_l|-1$ colors.
Since $(\bigcup_{w\in V(A')}\mathcal{T}_w)\cap \mathcal{T}_u=\emptyset$,
$\mathcal{T}_u \cup(\bigcup_{w\in V(A')}\mathcal{T}_w)$ wastes at least $n-k+1$ colors, a contradiction.
Therefore, $|\bigcup_{l\neq i}A_l|\geq 3$ for each $i\in[r]$, i.e., $|A_i|=2$ for $i\in[r]$ and $r=2$.
Let $V(A_1)=\{x_1,x_2\}$ and $V(A_2)=\{y_1,y_2\}$.
Then each nontrivial tree contains at most two of $\{x_1,x_2,y_1,y_2\}$.
Therefore, there is a nontrivial tree $T_{i,j}$ connecting $x_i,y_j$ for $i,j\in[2]$,
and the four nontrivial trees are pairwise different.
Since $n=k+4$ in this case and $\Gamma$ wastes $n-k=4$ colors, each $T_{i,j}$ is a $2$-path and there is no other nontrivial tree.
By Claim \ref{Ex-clm00}, the two leaves of each $T_{i,j}$ are nonadjacent.
Thus, $\overline{G}=\{x_1y_1,x_1y_2,x_2y_1,x_2y_2\}$ is a $4$-cycle.
Choose a vertex of $S$, say $v_1$.
Let $T=\bigcup_{i\in[2]}(v_1x_i\cup v_1y_i)$.
Then $T$ is a tree of $G$.
Let $\Gamma'$ be an edge-coloring of $G$ such that $T$ is the only nontrivial tree.
Then $\Gamma'$ is an MC-coloring of $G$ and it wastes three colors,
which contradicts that $\Gamma$ is an extremal MC-coloring of $G$.
\end{proof}

\section{Classification of planar graphs}

In this section, we consider the monochromatic connection numbers of all planar graphs.
Since the connectivity of a planar graph is at most five,
its monochromatic connection number is less than or equal to $m-n+6$.
In fact, $ m-n+2\leq mc(G)\leq m-n+4$ if $G$ is a planar graph.
We depict all $k$-connected but not $(k+1)$-connected planar graphs $G$ with $mc(G)=m-n+r$,
where $1\leq k\leq 5$ and $2\leq r\leq 4$.

It is well-known that a graph is outerplanar if and only if it does not contain a $K_4$-minor or a $K_{2,3}$-minor,
and a $2$-connected outerplanar graph contains a $2$-degree vertex.
Moreover, the exterior face of an outerplanar graph $G$ is a Hamiltonian cycle (called the {\em boundary} of $G$).

\begin{lemma}\label{Pla-vee}
Let $H$ be a simple graph and $v$ an additional vertex. Then
\begin{enumerate}
\item $v\vee H$ is a planar graph if and only if $H$ is an outerplanar graph.
\item $2K_1\vee H$ is a planar graph if and only if $H$ is either a cycle or linear forest.
\item $K_2\vee H$ is a planar graph if and only if $H$ is a linear forest.
\item if $H$ is a $2$-connected outerplanar graph with $|H|\geq 4$, then $H$ contains two nonadjacent $2$-degree vertices.
\end{enumerate}
\end{lemma}
\begin{proof}
Notice that $v\vee H$ is a planar graph if $H$ is an outerplanar graph.
On the other hand, if $v\vee H$ is a planar graph but $H$ is not an outerplanar graph, then $H$ contains either a $K_4$-minor or a $K_{2,3}$-minor.
Therefore, $v\vee H$ contains either a $K_5$-minor or a $K_{3,3}$-minor, a contradiction.

It is obvious that $2K_1\vee S_3$ contains a $K_{3,3}$ as its subgraph, and $2K_1\vee (K_3+K_1)$ contains a $K_5$-minor.
Therefore, $H$ does not have vertices of degree greater than or equal to three when $2K_1\vee H$ is a planar graph,
i.e., each component of $H$ is either a cycle or a path.
If $H$ has two components $H_1,H_2$ such that $H_1$ is a cycle, then $H$ has a $(K_3+K_1)$-minor.
Thus, $2K_1\vee H$ has a $K_5$-minor, a contradiction.
Therefore, $H$ is either a cycle or a linear forest if $2K_1\vee H$ is a planar graph.
On the other hand, if  each component of $H$ is either a cycle or a linear forest, then $2K_1\vee H$ is clearly a planar graph.

If $H$ is a linear forest, then $K_2\vee H$ is obviously a planar graph.
If $K_2\vee H$ is a planar graph, then $H$ is either a cycle or a linear forest, since $2K_1\vee H$ is a subgraph of $K_2\vee H$.
Since $K_2\vee H$ contains a $K_5$-minor if one component of $H$ is a cycle, $H$ is a linear forest.

If $H$ is a $2$-connected outerplanar graph with $|H|=4$, then $H$ has two nonadjacent $2$-degree vertices.
If $|H|\geq 5$ and $H$ does not have chord, then $H$ has two nonadjacent $2$-degree vertices.
If $|H|\geq 5$ and $H$ has a chord $e=xy$, then the two $\{x,y\}$-components, say $H_1$ and $H_2$, are $2$-connected outerplanar graphs.
For $i\in[2]$, if $|H_i|\geq 4$, then by induction, $H_i$ has a vertex $z_i\notin\{x,y\}$ such that $d_{H_i}(z_i)=2$;
if $H_i=K_3$, let $\{z_i\}=V(H_i)-\{x,y\}$.
Then $z_1,z_2$ are two nonadjacent $2$-degree vertices of $H$.
\end{proof}

Let $\mathcal{P}_1$ denote the set of graphs $G=v\vee H$, where $H$ is a connected outerplanar graph with a cut-vertex.

\begin{lemma}
If $G$ is a $2$-connected but not $3$-connected planar graph, then $mc(G)=m-n+3$ if and only if $G \in\mathcal{P}_1$.
\end{lemma}
\begin{proof}
By Lemma \ref{Pla-vee} (1) and Theorem \ref{Ex-mc-thm1}, $G$ is a planar graph and $mc(G)=m-n+3$ if $G\in \mathcal{P}_1$.
Suppose $mc(G)=m-n+3$.
Then by Theorem \ref{Ex-mc-thm1}, $G$ is either a $2$-perfectly-connected graph or a graph in $\mathcal{A}_{n,2}$.
If $G\in \mathcal{A}_{n,2}$,
then $G=v\vee H$ and $H$ is a connected graph with a cut-vertex.
Then by Lemma \ref{Pla-vee} (1), $H$ is a connected outerplanar graph with a cut-vertex.
If $G$ is a $2$-perfectly-connected graph, then $V(G)$ can be partitioned into three nonempty sets $\{v\},A,B$ such that $A,B$ form a complete bipartite graph.
Let $|A|\leq |B|$.
Then $1\leq |A|\leq 2$ (otherwise $G$ contains a $K_{3,3}$ as its subgraph).
If $|A|=1$ (say $A=\{x\}$), then by Lemma \ref{Pla-vee} (1), $G[B]$ is a connected outerplanar graph.
Let $H=G[B\cup v]$.
Then $H$ is a connected outerplanar graph with a cut-vertex and $G=x\vee H$, i.e., $G\in \mathcal{P}_1$.
If $|A|=2$, i.e., $G[A]=K_2$, then $G[B]$ is a path by Lemma \ref{Pla-vee} (3).
Let $A=\{x,y\}$ and $N(v)=\{x,z\}$,
Then $G-x=(y\vee G[B])\cup vz$.
Since $G[B]$ is a path, $G-x$ is an outerplanar graph with a cut-vertex $z$.
Since $G=x\vee (G-x)$, we get $G\in \mathcal{P}_1$.
\end{proof}

Let $\mathcal{P}_2=\{v\vee H:~H\mbox{ is a $2$-connected outerplanar graph and }H\neq u\vee P_{n-2}\}.$

\begin{lemma}
If $G$ is a $3$-connected but not $4$-connected planar graph, then $mc(G)=m-n+3$ if and only if either $G=2K_1\vee P_{n-2}$ or $G\in\mathcal{P}_2$,
and $mc(G)=m-n+4$ if and only if $G=K_2\vee P_{n-2}$.
\end{lemma}
\begin{proof}
By Lemma \ref{Pla-vee} (3) and Theorem \ref{Ex-mc-thm1}, $K_2\vee P_{n-2}$ is a planar graph with $mc(K_2\vee P_{n-2})=m-n+4$.
Next, we prove that $G=K_2\vee P_{n-2}$ if  $mc(G)=m-n+4$.
Suppose $mc(G)=m-n+4$.
Then either $G\in \mathcal{A}_{n,3}$ or $G$ is a $3$-perfectly-connected graph.
If $G$ is the latter, then $V(G)$ can be partitioned into four parts $v,V_1,V_2,V_3$, such that each $V_i$ induces a connected subgraph,
$V_1,V_2,V_3$ form a complete 3-partite graph, and $v$ has precisely one neighbor in each $V_i$.
Let $|V_1|\leq |V_2|\leq |V_3|$.
If $|V_1|=|V_2|=1$, then $G[V_1\cup V_2]$ is an edge, say $e$.
Thus, $G=e\vee G[V_3\cup v]$.
By Lemma \ref{Pla-vee} (3), since $G$ is a $3$-connected graph, $G[V_3\cup v]$ is a path of order $n-2$.
Therefore, $G=K_2\vee P_{n-2}$.
If $|V_2|\geq 2$, then $G[V_1\cup V_2\cup V_3]$ contains a $K_5$-minor, a contradiction.
If $G$ is the former, i.e., $G\in \mathcal{A}_{n,3}$, then $G=K_2\vee H$.
By Lemma \ref{Pla-vee} (3), since $G$ is a $3$-connected graph, $G=K_2\vee P_{n-2}$.
Therefore, $mc(G)=m-n+4$ if and only if $G=K_2\vee P_{n-2}$.

If $mc(G)=m-n+3$, then $G\in \mathcal{B}^1_{n,3}\cup \mathcal{B}^2_{n,3}\cup \mathcal{B}^3_{n,3}$.
If $G\in \mathcal{B}^3_{n,3}$, then $V(G)$ can be partitioned into two parts $U,V$ such that $G[U]=K_2^-=2K_1$, $G[V]$ is a connected graph and $G=G[U]\vee G[V]$.
By Lemma \ref{Pla-vee} (2), since $G$ is a $3$-connected graph, $G[V]$ is either a cycle or a path.
Since $G$ is not a $4$-connected planar graph, $G[V]$ is a path, i.e., $G=2K_1\vee P_{n-2}$.
If $G\in\mathcal{B}^2_{n,3}$, then $G=K_1\vee H$,  where $H$ is a $2$-connected but not $3$-connected graph.
Since $G$ is planar, by Lemma \ref{Pla-vee} (1), $H$ is a $2$-connected outerplanar graph (recall that the outerplanar graph is not $3$-connected).
Therefore, $G\in \mathcal{P}_2$.
If $G\in \mathcal{B}^1_{n,3}$, then $V(G)$ can be partitioned into three parts $v,A,B$, such that $v$ has two neighbors in $A$ and one neighbor
in $B$, and $A,B$ form a complete bipartite graph.

If $G[A]=K_2$, then by Lemma \ref{Pla-vee} (3), $G[B]$ is a path $P_{n-3}$.
Thus, $G=K_2\vee P_{n-2}$, a contradiction to the assumption that $mc(G)=m-n+3$.
If $G[A]=2K_1$, then $G=G[A]\vee G[B\cup v]$.
By Lemma \ref{Pla-vee} (2), $G[B\cup v]$ is either a path $P_{n-3}$ or a cycle $C_{n-3}$.
Since $v$ has precisely one neighbor in $B$, $G[B\cup v]$ is a path.
Thus, $G=2K_1\vee P_{n-2}$.

If $|A|\geq 3$, then $|B|\leq 2$.
Let $x$ be the neighbor of $v$ in $B$.
Since $mc(G)=m-n+3$, we have $G\neq K_2\vee P_{n-2}$.
If $|B|=2$, i.e., $G[B]=K_2$, then $G=x\vee(G-x)$.
Thus, $G-x$ is a $2$-connected outerplanar graph.
If $|B|=1$, then $V(B)=\{x\}$ and $G=x\vee (G-x)$, and thus $G-x$ is a $2$-connected outerplanar graph.
Therefore, $G\in\mathcal{P}_2$.
\end{proof}

\begin{claim}\label{Ex-claim1}
Suppose $G$ is a $k$-connected planar graph and $S$ is a vertex-cut with $|S|=k$.
If $k\geq 4$, then $G[S]$ does not contain the vertices of degree greater than two.
\end{claim}
\begin{proof}
Let $u,v$ be two vertices in different components of $G-S$.
Since $G$ is a $k$-connected graph, there are $k$ internally disjoint $uv$-paths $L_1,\cdots,L_k$.
Let $H$ be a graph obtained from $\bigcup_{i\in[k]}L_i$ by contracting all edges but those incident with $u$ and $v$.
Then $H=K_{2,k}$ is a minor of $G$ with one part $S$.
Thus, by Lemma \ref{Pla-vee} (2), $G[S]$ does not contain the vertices of degree greater than two.
\end{proof}

\begin{claim}\label{Ex-claim2}
Let $G$ be a $k$-connected planar graph and $S$ be a vertex-cut with $|S|=k$.
Suppose $\Gamma$ is an extremal MC-coloring of $G$ such that $G[S]$ does not contain nontrivial edges.
Then
\begin{enumerate}
\item if $k=4$ and $G[S]$ is not a $4$-cycle, then $mc(G)=m-n+2$;
\item if $k=5$, then $mc(G)=m-n+2$.
\end{enumerate}
In addition, if $k=4$ and $G[S]$ does not contain nontrivial edges under any extremal MC-colorings, then $mc(G)=m-n+2$.
\end{claim}
\begin{proof}
We first prove (1) and (2).
By Claim \ref{Ex-claim1}, $G$ has a $K_{2,k}$-minor with one part $S$.
Since $G$ is a planar graph, by Lemma \ref{Pla-vee} (2), $G[S]$ is either a cycle or a linear forest.
Thus, $\overline{G[S]}$ contains a $5$-cycle if $k=5$.
For $k=4$, $\overline{G[S]}$ contains a $P_4$ if $G[S]\neq C_4$.
Suppose $A_1,\cdots,A_r$ are the components of $G-S$.

Let $\Gamma$ be an extremal MC-coloring of $G$.
We use $\mathcal{S}$ to denote the set of all nontrivial trees of $G$.
Choose two vertices $u,v$ from $A_1,A_2$, respectively.
Let $U=V(A_1)-V(\mathcal{T}_u)$ and $\mathcal{F}=\mathcal{T}_v-\mathcal{T}_u$.
Assume $\mathcal{T}=\mathcal{T}_u\cup \mathcal{F}$.
For each $T\in \mathcal{S}$, let $x_T=|V(T)\cap S|$ when $|V(T)\cap S|\geq 2$ and let $x_T=1$ when $|V(T)\cap S|\leq 1$.
Since $V(G)-S\subseteq V(\mathcal{T})$ and each tree of $\mathcal{T}$ contains at least one vertex of $S$, $\mathcal{T}$ wastes at least $n-k-1+\sum_{T\in \mathcal{T}}(x_T-1)$ colors.
Since $G[S]$ does not contain nontrivial edges, if $x_T\geq 2$, then $T$ wastes at least $x_T-1$ colors.
Then $\Gamma$ wastes $w_\Gamma\geq n-k-1+\sum_{T\in \mathcal{S}}(x_T-1)$ colors.
Let $T$ be a tree of $\mathcal{S}$ such that $x_T$ is maximum.

Suppose $x_T\geq 4$.
If $k=4$, then $w_\Gamma\geq n-2$.
If $k=5$ and $x_T\geq 5$, then $w_\Gamma\geq n-2$.
If $k=5$ and $x_T=4$, then let $S-V(T)=\{a\}$.
Since $\overline{G[S]}$ contains a $5$-cycle,
$a$ does connect a vertex of $S-a$ in $G[S]$.
Therefore, $a$ connects this vertex by a nontrivial tree different from $T$.
Thus, $w_\Gamma\geq n-2$.

Suppose $x_T=3$.
If $k=4$, then let $S-V(T)=\{a\}$.
Since $\overline{G[S]}$ contains a $P_4$,
$a$ connects a vertex of $S-a$ by a nontrivial tree.
Thus, $w_\Gamma\geq n-2$.
If $k=5$, then let $\{a,b\}=S-V(T)$.
Since $\overline{G[S]}$ contains a $5$-cycle,
$a$ connects a vertex of $S-a$ by a nontrivial tree $T_1$,
and $b$ connects a vertex of $S-a$ by a nontrivial tree $T_2$.
Whenever $T_1=T_2$ or not, $\Gamma$ wastes at least $n-2$ colors.

Suppose $x_T=2$.
Since $T$ is a tree of $\mathcal{S}$ such that $x_T$ is maximum,
for any two different nonadjacent vertex pairs of $S$, there are two different nontrivial trees connecting them, respectively.
If $k=4$, then since $\overline{G[S]}$ contains a $P_4$, $\Gamma$ wastes at least $n-k-1+3=n-2$ colors.
If $k=5$, then since $\overline{G[S]}$ is a $5$-cycle, $\Gamma$ wastes at least $n-k-1+5=n-1$ colors, which contradicts that $\Gamma$ is extremal.

Now we prove that if $k=4$ and $G[S]$ does not contain nontrivial edges under any extremal MC-colorings, then $mc(G)=m-n+2$.
Suppose $\Gamma$ is an extremal MC-coloring of $G$ and $T$ is a nontrivial tree with $x_T=|V(T)\cap S|$ maximum.
Similar to the above proof for $k=4$, we can obtain that $mc(G)=m-n+2$ except for the case that $G[S]$ is a $4$-cycle and $x_T=2$.
For the case that $G[S]$ is a $4$-cycle and $x_T=2$, let $E(\overline{G[S]})=\{v_1v_2,v_3v_4\}$. Then there is a nontrivial tree $T_1$ connecting $v_1,v_2$, and a nontrivial tree $T_2$ connecting $v_3,v_4$.
Suppose $mc(G)\geq m-n+3$.
Since $\Gamma$ wastes $n-k-1+(|T_1|-2)+(|T_2|-2)\leq n-3$,
$T_1$ and $T_2$ are $2$-paths.
Let $\Gamma'$ be an edge-coloring of $G$ obtained from $\Gamma$ by recoloring $E(T_1\cup T_2)$ by trivial colors and recoloring a $3$-path of $G[S]$ by a new nontrivial color.
Then $\Gamma'$ is an extremal MC-coloring of $G$ and $G[S]$ contains nontrivial edges under $\Gamma'$, a contradiction.
\end{proof}

\begin{claim}\label{Ex-claim3}
Let $\Gamma$ be a simple extremal MC-coloring of $G$ and $e=xy$ be a nontrivial edge in $G$.
Suppose $mc(G)=e(G)-|G|+x$ and $H$ is the underlying graph of $G/e$.
Then $mc(H)\geq e(H)-|H|+x$.
\end{claim}
\begin{proof}
Since $\Gamma$ is a simple extremal MC-coloring of $G$ and $mc(G)=e(G)-|G|+x$, $\Gamma$ wastes $|G|-x$ colors.
Suppose $z$ is the new vertex of $V(G/e)$.
Then any parallel edges are incident with $z$, and at most two parallel edges between two vertices.
Since $e$ is a nontrivial edge, $\Gamma$ is simple and every color-induced subgraph in $G$ is a tree, any color-induced subgraph of $G/e$ is a tree.
It is obvious that any two vertices of $G/e$ are connected by a monochromatic tree under $\Gamma|_{G/e}$.
Moreover, $\Gamma|_{G/e}$ wastes $|G|-1-x=|G/e|-x$ colors.

Suppose there are parallel edges $e_1,e_2$ between $u$ and $z$.
If there is a trivial and parallel edge between $u$ and $z$, say $e_1$, then we delete $e_1$.
Then the resulting graph is also monochromatic connected, and the edge-coloring wastes $|G/e|-x$ colors.
If the two parallel edges are nontrivial, then suppose the $e_1,e_2$ are edges of two nontrivial trees $T_1,T_2$, respectively.
Let $T$ be a spanning tree of $T_1\cup T_2$ containing $e_1$.
Let $\Gamma'$ be an edge-coloring of $G/e-e_2$ obtained from $\Gamma$ by recoloring $T$ with a new nontrivial color, and then recoloring any other edges of $E(T_1\cup T_2)-E(T)-e_2$ with trivial colors.
Then $\Gamma'$ is an MC-coloring of $G/e-e_2$ and $\Gamma'$ wastes at most $|G/e-e_2|-x=|G/e|-x$ colors.
By the above operation, we obtain an underlying graph $H$ of $G/e$, and a simple MC-coloring $\Gamma''$ of $H$, which wastes at most $|H|-x$ colors.
Thus, $mc(H)\geq e(H)-|H|+x$.
\end{proof}

\begin{claim}\label{Ex-claim4}
Let $G$ be a planar graph and $e=ab$ be an edge of $G$.
If the underlying graph of $G/e$ contains $\{u,v\}\vee P_t$ as its subgraph, $u$ is the new vertices and $a$ (and also $b$) connects two leaves of $P_t$,
then either $N_G(a)\cap I=\emptyset$ and $I\subseteq N_G(b)$,
or $N_G(b)\cap I=\emptyset$ and $I\subseteq N_G(a)$, where $I$ is the set of internal vertices of $P_t$.
\end{claim}
\begin{proof}
If $N_G(a)\cap I\neq \emptyset$ and $N_G(b)\cap I\neq \emptyset$,
then let $G'$ be a graph obtained from $G$ by contracting all but two pendent edges of $P_t$.
Then $G'$ has a $K_{3,3}$ with one part $\{a,b,v\}$, i.e., $G$ has a $K_{3,3}$-minor, a contradiction.
\end{proof}

\begin{lemma}
If $G$ is a $4$-connected but not $5$-connected planar graph, then $mc(G)\leq m-n+3$,
and $mc(G)=m-n+3$ if and only if $G=2K_1\vee C_{n-2}$.
\end{lemma}
\begin{proof}
Suppose $G=\{u,v\}\vee H$, where $H$ is a $(n-2)$-cycle and $uv$ is not an edge of $G$.
Then there is a $2$-path $P$ connecting $u$ and $v$.
Let $L$ be a spanning tree of $H$.
Let $\Gamma$ be an edge-coloring such that $P$ and $L$ are all nontrivial trees of $G$.
Then $\Gamma$ is an MC-coloring of $G$ , which wastes $n-3$ colors.
Thus, $mc(G)\geq m-n+3$.
It is easy to verify that $G$ is neither a graph of $\mathcal{A}_{n,4}\cup \mathcal{B}^1_{n,4}\cup \mathcal{B}_{n,4}^2\cup \mathcal{B}_{n,4}^3$, nor a $4$-perfectly-connected graph.
Therefore, $mc(G)=m-n+3$.

Suppose $mc(G)\geq m-n+3$.
We prove that $G=2K_1\vee C_{n-2}$ below.
Suppose $S=\{x_1,x_2,x_3,x_4\}$ is a vertex-cut of $G$.
If $G[S]$ does not contain nontrivial edges under any extremal MC-colorings of $G$,
then by Claim \ref{Ex-claim2}, $mc(G)=m-n+2$.
If there is an extremal MC-coloring $\Gamma$ of $G$ such that $G[S]$ has a nontrivial edge, say  $e=x_1x_2$,
then by Lemma \ref{Ex-claim3} the underlying graph $H$ of $G/e$ satisfies that $mc(H)\geq e(H)-|H|+3$.
Since $H$ is a $3$-connected but not $4$-connected graph, $H$ is either $2K_1\vee P_{n-3}$ or $K_2\vee P_{n-3}$, or a graph of $\mathcal{P}_2$.
Since $G$ is a $4$-connected graph,
if there is a vertex $x$ of $H$ with $d_H(x)=3$, then $x$ is incident with the new vertex.

{\bf Case 1.} Either $H=2K_1\vee P_{n-3}$ or $H=K_2\vee P_{n-3}$.

From the assumption, $V(H)$ can be partitioned into two parts $A=\{u,v\}$ and $B$,
such that $H[B]=P_{n-3}$ and $H=H[A]\vee H[B]$. Here, $uv$ is an edge of $H$ if $H=K_2\vee P_{n-3}$,
and $uv$ is not an edge of $H$ if $H=2K_1\vee P_{n-3}$.
Let $H[B]=v_1e_1v_2e_2\cdots e_{n-4}v_{n-3}$.
If $|B|=3$, then $H=K_1\vee C_4$.
Since each vertex of $V(H)-\{v_2\}$ has degree three in $H$, $v_2$ is the new vertex and $G=K_2\vee C_4$,
a contradiction to the choice of $G$ as a planar graph.
Thus, $|B|\geq 4$ and $v_1,v_{n-3}$ are the only two vertices with degree 3 in $H$.
Therefore, the new vertex is either $u$ or $v$ (by symmetry, say $u$).
Since $G$ is a $4$-connected graph, $v_1$ (and also $v_{n-3}$) connects $x_1,x_2$ in $G$.
Then by Claim \ref{Ex-claim4}, suppose $x_1$ does not connect any vertices of $\{v_2,\cdots,v_{n-4}\}$ and
$x_2$ connects every vertex of $\{v_2,\cdots,v_{n-4}\}$.
Since $G$ is a $4$-connected graph,  $x_1$ connects $v$.
Then $G[B\cup x_1]$ is a cycle and thus $G=2K_1\vee C_{n-2}$.

{\bf Case 2.} $H\in \mathcal{P}_2$.

From the definition of $\mathcal{P}_2$, $H=v\vee R$, where $R$ is a $2$-connected outerplanar graph.
If $R=K_3$, then $|G|=5$.
Since $G$ is a $4$-connected graph, $G=K_5$, a contradiction.
Thus, $|R|\geq 4$.
Since $R$ is a $2$-connected outerplanar graph, by Lemma \ref{Pla-vee} (4), $R$ has two nonadjacent $2$-degree vertices.
Moreover, the boundary $C$ of $R$ is its Hamiltonian cycle.

{\bf Case 2.1.} $R$ has at least three vertices of degree two, say $u_1,u_2,u_3$.

Note that every $2$-degree vertex of $R$ is incident with the new vertex in $H$.
Thus, $v$ is the new vertex and each $u_i$ connects both $x_1$ and $x_2$ in $G$.
Note that $u_1,u_2$ and $u_3$ divide $C$ into three paths.
Let $H'$ be a graph obtained from $H$ by contracting all but one edge of each such path.
Then the underlying graph of $H'$ is $K_5$, i.e., $G$ has a $K_5$-minor, a contradiction.

{\bf Case 2.2.} $R$ has exactly two vertices of degree two and $v$ is not the new vertex.

Suppose $w_1,w_2$ are $2$-degree vertices of $R$.
Since $v$ is not the new vertex,
$w_1,w_2$ have a common neighbor $z$ in $R$,
and $z$ is the new vertex.

Let $P=R-z$.
We prove that $H=vz\vee P$ and $P$ is a path.
We first prove that $R=z\vee P$, i.e., each chord of $R$ is incident with $z$.
Suppose to the contrary, there is a chord $f=z_1z_2$ of $R$ such that $z\notin \{z_1,z_2\}$.
Then $z_1,z_2$ divide $C$ into two paths $L_1$ and $L_2$, say $z$ is an internal vertex of $L_1$.
Since $R$ is an outerplanar graph, $z$ does not connect any internal vertices of $L_2$ in $H$.
Furthermore, since $z$ is the new vertex, neither $x_1$ nor $x_2$ connects internal vertices of $L_2$ in $G$.
Thus, $\{v,z_1,z_2\}$ is a vertex-cut of $G$, a contradiction to the assumption that $G$ is a $4$-connected graph.
So, $R=z\vee P$ and $P$ is a path.
Since $v$ connects every vertex of $R$, we have $H=vz\vee P$.

Consider the graph $G$ below.
Since $w_1,w_2$ are $3$-degree vertices and $z$ is the new vertex of $H$, $w_1$ (and also $w_2$) connects $x_1$ and $x_2$ in $G$.
Let $I=V(P)-\{w_1,w_2\}$.
Since $H=vz\vee P$, by Claim \ref{Ex-claim4}, suppose $x_1$ does not connect any vertices of $I$ and $x_2$ connects every vertex of $I$.
Then $D=G[V(P)\cup x_1]$ is a $C_{n-2}$ and $G-v=x_2\vee D$.
Since $\{v,x_2\}\vee D$ is a spanning subgraph of $G$, $v$ does not connect $x_2$ by Lemma \ref{Pla-vee} (3).
This implies $G=\{x_2,v\}\vee D$, i.e., $G=2K_1\vee C_{n-2}$.

{\bf Case 2.3.} $R$ has exactly two vertices of degree two and $v$ is the new vertex.

Suppose $a,b$ are nonadjacent $2$-degree vertices of $R$.
Then $a,b$ divide $C$ into two paths, say $L_1$ and $L_2$.
Let $L_1=ae_1z_1e_2,\cdots z_se_{s+1}b$ and
$L_2=af_1w_1f_2,\cdots w_tf_{t+1}b$.

If $N_G(x_1)\cap (V(L_1)-\{a,b\})\neq \emptyset$ and $N_G(x_2)\cap (V(L_1)-\{a,b\})\neq \emptyset$,
then let $H'$ be a graph obtained from $H$ by contracting all edges of $C$ but $e_1,e_{s+1}$ and $f_1$.
Then the underlying graph of $H'$ is $K_5$, i.e., $G$ has a $K_5$-minor, a contradiction.
Thus, by symmetry, suppose $V(L_1)-\{a,b\}\subseteq N_G(x_1)$ and $N_G(x_2)\cap (V(L_1)-\{a,b\})\}= \emptyset$.
By the same reason, $N_G(x_1)\cap (V(L_2)-\{a,b\})\}\neq \emptyset$ and $N_G(x_2)\cap (V(L_2)-\{a,b\})\}\neq \emptyset$ will not happen.
Since $H$ is a $3$-connected graph,
$V(L_2)-\{a,b\}\subseteq N_G(x_2)$ and $N_G(x_1)\cap (V(L_2)-\{a,b\})\}=\emptyset$.
Therefore, $N_G(a)\cap V(R)=V(L_1)$ and $N_G(b)\cap V(R)=V(L_2)$.

If $R=K_1\vee P_{n-3}$, then $G=2K_1\vee C_{n-2}$.
Thus, we only need to prove that $R=K_1\vee P_{n-3}$ below.
\begin{claim}\label{Ex-clm6}
Suppose $l=n_1n_2$ is a chord of $R$. Then one end of $l$ is contained in $V(L_1)-\{a,b\}$ and the other end of $l$ is contained in $V(L_2)-\{a,b\}$.
\end{claim}
\begin{proof}
Suppose, to the contrary, $\{n_1,n_2\}\subseteq V(L_1)$.
Then $S'=\{x_1,x_2,n_1,n_2\}$ is a vertex-cut of $G$ with $|S'|=4$. However, $d_{G[S']}(x_1)=3$, a contradiction to Claim \ref{Ex-claim1}.
\end{proof}

If, by symmetry, $|L_1|=3$, i.e., $L_1=ae_1z_1e_2b$, then by Claim \ref{Ex-clm6},
$z_1$ connects every vertex of $L_2$.
Thus, $R=K_1\vee P_{n-3}$.

If $|L_1|,|L_2|\geq 4$.
Recall that $e=x_1x_2$ is a nontrivial edge under $\Gamma$.
Suppose $e$ is an edge of a nontrivial tree $T$.
Then there is a nontrivial edge $f$ of $T$ between $\{x_1,x_2\}$ and $R$.
By symmetry, suppose one end of $f$ is $x_1$ and the other end of $f$ is contained in $V(L_1)$.
Suppose $H'$ is the underlying graph of $G/f$.
Then $mc(H')\geq e(H')-|H'|+3$.
Since $H'$ is a $3$-connected planar graph,
$H'$ is either $2K_1\vee P_{n-3}$ or $K_2\vee P_{n-3}$, or a graph of $\mathcal{P}_2$.

Suppose $H'$ is either $2K_1\vee P_{n-3}$ or $K_2\vee P_{n-3}$.
Let $H'=A\vee P_{n-3}$, where $V(A)=\{y_1,y_2\}$.
If $x_2\in\{y_1,y_2\}$, say $x_2=y_2$, then $y_1$ connects every vertex of $V(R)-\{y_1\}$.
Thus, either $|L_1|=3$ or $|L_2|=3$, a contradiction.
If $x_2\notin\{y_1,y_2\}$, then $R=\{y_1,y_2\}\vee (R-y_1-y_2)$.
Since $|R|\geq 6$, we have $|R-y_1-y_2|\geq 4$.
Thus, $R$ has a $K_{2,3}$-minor, which contradicts that $R$ is an outerplanar graph.

Suppose $H'$ is a graph of $\mathcal{P}_2$.
Then $H'=y\vee H''$, where $H''$ is a $2$-connected outerplanar graph.
If $y=x_2$, then $x_2$ connects every vertex of $R$.
However, since $|L_1|\geq 4$ and $x_2$ does not connect any internal vertex of $L_1$ in $G$,
there is an internal vertex of $L_1$ does not connect $x_2$ in $H'$, a contradiction to the fact that $H'=x_2\vee H''$.
If $y\neq x_2$, then $y\in V(R)$ and thus $R=K_1\vee P_{n-3}$, a contradiction to the assumption that $|L_1|,|L_2|\geq 4$.
\end{proof}

\begin{lemma}
If $G$ is a $5$-connected planar graph, then $mc(G)=m-n+2$.
\end{lemma}
\begin{proof}
Suppose $mc(G)\geq m-n+3$.
Let $S=\{v_1,\cdots,v_5\}$ be a vertex-cut of $G$.
If $G[S]$ does not contain nontrivial edges under any extremal MC-colorings of $G$, then by Claim \ref{Ex-claim2}, $mc(G)=m-n+2$, a contradiction.
Otherwise, there is a nontrivial edge in $G[S]$, say $e=v_1v_2$.
Let $H$ be the underlying graph of $G/e$.
Then by Claim \ref{Ex-claim3}, $mc(H)\geq e(H)-|H|+3$.
Since $H$ is a $4$-connected but not $5$-connected graph, we have $mc(H)=e(H)-|H|+3$.
Thus, $H=2K_1\vee C_{n-2}$, say $H=\{u,v\}\vee C$, where $C =C_{n-2}$.
Since each vertex of $C$ has degree 4 in $H$, either $u$ or $v$ is the new vertex.
By symmetry, let $u$ be the new vertex.
Thus, $v_1,v_2$ connect every vertex of $C$, i.e., $e\vee C$ is a subgraph of $G$, a contradiction to the choice that $G$ is planar.
\end{proof}

\begin{theorem}
Suppose $G$ is a connected planar graph. Then $mc(G)\leq m-n+4$ and the following results hold.
\begin{enumerate}
\item If $G$ is not a $2$-connected graph, then $mc(G)=m-n+2$;
\item if $G$ is a $2$-connected but not $3$-connected graph, then $m-n+2\leq mc(G)\leq m-n+3$ and $mc(G)=m-n+3$ if and only if $G\in \mathcal{P}_1$;
\item if $G$ is a $3$-connected but not $4$-connected graph, then $m-n+2\leq mc(G)\leq m-n+4$. Moreover, $mc(G)=m-n+4$ if and only if $G=K_2\vee P_{n-2}$,
    and $mc(G)=m-n+3$ if and only if either $G\in \mathcal{P}_2$, or $G=2K_1\vee P_{n-2}$;
\item if $G$ is a $4$-connected but not $5$-connected graph, then $m-n+2\leq mc(G)\leq m-n+3$, and $mc(G)=m-n+3$ if and only if $G=2K_1\vee C_{n-2}$;
\item if $G$ is a $5$-connected graph, then $mc(G)=m-n+2$.
\end{enumerate}
\end{theorem}

For ease of reading, the classification of planar graphs are summarized in the following table (remember that the connectivity $\kappa(G)$
of a planar graph $G$ is at most 5).

\begin{table}[!htbp]
\centering
\begin{tabular}{|c|c|c|c|c|c|}
\hline\backslashbox{$mc(G)$\kern-2em}{$\kappa(G)$}
  &$1$ &  $2$ &  $3$ &  $4$ & $5$\\
 \hline
 $m-n+4$ &$\emptyset$ & $\emptyset$ &$G=K_2\vee P_{n-2}$ & $\emptyset$ & $\emptyset$ \\
 \hline
 $m-n+3$ & $\emptyset$ & $G\in\mathcal{P}_1$ &\tabincell{c}{either $G\in\mathcal{P}_2$,\\or $G=2K_1\vee P_{n-2}$}& $G=2K_1\vee C_{n-2}$ & $\emptyset$\\
 \hline
 $m-n+2$ & all & all but the above & all but the above & all but the above & all \\
 \hline
\end{tabular}
\caption{The classification of planar graphs.}
\end{table}


\begin{thebibliography}{1}

\bibitem {BL}
X. Bai, X. Li, Graph colorings under global structural conditions, arXiv:2008.07163 [math.CO].

\bibitem {B}
J.A. Bondy, U.S.R. Murty, Graph Theory, GTM 244, Springer,
2008.

\bibitem{CLW} Q. Cai, X. Li, D. Wu, Some extremal results on the colorful monochromatic vertex-connectivity
of a graph, J. Comb. Optim. 35(2018), 1300--1311.

\bibitem{CY} Y. Caro, R. Yuster, Colorful monochromatic connectivity, Discrete Math. 311(2011),
1786--1792.

\bibitem{GLQZ} R. Gu, X. Li, Z. Qin, Y. Zhao, More on the colorful monochromatic connectivity, Bull.
Malays. Math. Sci. Soc. 40(4)(2017), 1769--1779.

\bibitem{JLW} Z. Jin, X. Li, K. Wang, The monochromatic connectivity of graphs, Taiwanese J. Math. 24(4)(2020), 785--815.

\bibitem{JLY} Z. Jin, X. Li, Y. Yang, Extremal graphs with maximum monochromatic connectivity, Discrete Math. 343(9)(2020), 111968.

\bibitem{LL} P. Li, X. Li, Monochromatic $k$-edge-connection colorings of graphs, Discrete Math. 343(2)(2019), 111679.

\bibitem{LL1} P. Li, X. Li, Rainbow monochromatic $k$-edge-connection colorings of graphs, arXiv:2001.01419 [math.CO].

\bibitem{LW2} X. Li, D. Wu, A survey on monochromatic connections of graphs, Theory \& Appl. Graphs 0(1)(2018), Art.4.

\bibitem{MWYY} Y. Mao, Z. Wang, F. Yanling, C. Ye, Monochromatic connectivity and graph products,
Discrete Math, Algorithm. Appl. 8(01)(2016), 1650011.

\end{thebibliography}
\end{document}